\documentclass[11pt,english]{amsart}
\usepackage{textcomp}
\usepackage[utopia]{mathdesign}
\usepackage[T1]{fontenc}
\usepackage[a4paper]{geometry}
\usepackage{amsthm}
\usepackage{amstext}
\usepackage{setspace}
\makeatletter

\newcommand{\noun}[1]{\textsc{#1}}
\providecommand{\tabularnewline}{\\}

\theoremstyle{plain}
\newtheorem{thm}{Theorem}
  \theoremstyle{plain}
  \newtheorem{lem}{Lemma}

\makeatother

\usepackage{babel}

\begin{document}

\title{A Löwner variational method in the theory of schlicht functions}

\author{Eberhard Michel}

\maketitle
\noindent \begin{center}
Sonnenblumenweg 5
\par\end{center}

\noindent \begin{center}
65201 Wiesbaden
\par\end{center}
\begin{abstract}
A direct variational method is developed that allows to calculate
arbitrary continuous coefficient functionals of the second, third
and fourth coefficients of schlicht functions. Based on this method
an improved lower bound for the Milin-constant(0.034856..) is given,
as well as an improved lower bound for the maximal modulus of the
seventh coefficient of odd schlicht functions(1.006763..). 
\end{abstract}
{\small AMS Subject Classification}\noun{\small : }{\small 30C50 Coefficient
problems of univalent and multivalent functions}\\
{\small }

First set up some notation, denote the unit disk by $\mathbb{D}$,
the set of schlicht functions by S and for an arbitrary $n\in\mathcal{\mathbb{N}}$
let $\mathbb{V}_{n}$ denote the n-th coefficient region of the schlicht
functions. Consider a continuous coefficient functional $\Phi:\mathbb{V}_{n}\mathbb{\rightarrow C}$.
For every point $a\in\mathbb{V}_{n}$ there exists a function $f\in$
S so that $a=(a_{\text{\emph{2}}}(f),..,a_{\text{\emph{n}}}(f))$
where $a_{\text{\emph{k}}}(f)\; k=2,..,n$ denotes the k-th coefficient
of $f\in$ S. Then since S and $\mathbb{V}_{n}$ are compact(see \cite{1},
chapter 11)\begin{equation}
\mathit{max}\{\mathit{Re}\Phi(a)\,|\, a\in\mathbb{V}_{n}\}=\underset{f\in S}{\mathit{max}}\, Re\Phi(f)\end{equation}
exists. For the sake of simplicity the notations in (1) will be used
synonymously. The following results about the Löwner differential
equation can be found for instance in \cite{1} or \cite{2}. For
an arbitrary schlicht function $f\in$ S which is a single slit mapping
let $f(z)=z+\sum_{n=2}^{\infty}a_{n}z^{n}$. Then $f$ can be embedded
as initial point into a Löwner chain $f(z,t)=e^{\text{\emph{t}}}z+\sum_{n=2}^{\infty}a_{n}(t)z^{n}$,
i.e. $f(z,0)=f(z)$, $z\in D$ which satisfies the differential equation\begin{equation}
\dot{f}(z,t)=\frac{1+\kappa(t)z}{1-\kappa(t)z}zf'(z,t)\end{equation}
The Löwner differential equation (2) for the schlicht function $f:\mathbb{D}\rightarrow\mathbb{C}$
yields the differential equation\begin{equation}
\dot{a}_{\text{\emph{n}}}(t)=na_{n}(t)+\sum_{k=1}^{n-1}2ka_{k}(t)\kappa(t)^{\text{\emph{n-k}}}\end{equation}
or in integrated form\begin{equation}
a_{n}(t)=-e^{nt}\int_{\text{\emph{t}}}^{\infty}e^{\text{\emph{-ns}}}\sum_{k=1}^{n-1}2ka_{k}(s)\kappa(s)^{\text{\emph{n-k}}}ds\end{equation}
for the coefficients $a_{\text{\emph{n}}}(t)$ , $a_{n}(0)=a_{n}$
if $n\geq2$. For a single slit mapping $\kappa:[0,\infty)\rightarrow\mathbb{C}$
is a continuous function with $|\kappa(t)|=1.$ Substituting $x=e^{\text{\emph{-t}}}$
and putting $\alpha_{n}(x)=a_{n}(t)$ yields the equivalent differential
equation\begin{equation}
-x\alpha_{\text{\emph{n}}}'(x)=n\alpha_{\text{\emph{n}}}(x)+\sum_{k=1}^{n-1}2k\alpha_{k}(x)\mu(x)^{\text{\emph{n-k}}}\end{equation}
for $n\geq2$ where $\mu:(0,1]\rightarrow\mathbb{C}$ is a continuous
function that satisfies $\mu(x)=\kappa(t).$ Let $g_{\text{\emph{n}}}(x)=x^{\text{\emph{n}}}\alpha_{\text{\emph{n}}}(x)$
and $g_{1}(x)\equiv1$ then (5) takes the form\begin{equation}
g_{\text{\emph{n}}}'(x)=-\sum_{k=1}^{n-1}2kx^{\text{\emph{n-k-1}}}g_{k}(x)\mu(x)^{\text{\emph{n-k}}}\end{equation}
or in integrated form\begin{equation}
g_{\text{\emph{n}}}(x)=-\int_{\emph{0}}^{\text{\emph{x}}}\sum_{k=1}^{n-1}2kt^{\text{\emph{n-k-1}}}g_{k}(t)\mu(t)^{\text{\emph{n-k}}}dt\end{equation}
if $n\geq2$. Notice that $g_{\text{\emph{n}}}(1)=a_{\text{\emph{n}}}$
and $g_{\text{\emph{n}}}(0)=0$ if $n\geq2$. These representations
hold in particular if $\mu$ is continuous or a step function with
a finite number of steps. In the sequel the function $\mu:(0,1]\rightarrow\mathbb{C}$(and
$\kappa$ synonymously) will be called a generating function of the
schlicht function $f:\mathbb{D}\rightarrow\mathbb{C}$, since because
of (7) and the identity theorem it uniquely determines the coefficients
of the schlicht function $f$. No assumption whether the function
$f\in$ S uniquely determines the generating function is needed in
the sequel. \\
Some further preparations are necessary in order to formulate
Theorem 1, in particular a special class of step functions needs to
be introduced. For every $m\in\mathbb{N}$ consider the equidistant
partition of the interval $[0,1]$ into $m$ subintervals $I_{k}$,
$k=1,..,m$ and define $s:[0,1]\mathbb{\rightarrow C}$ by $s(x)=c_{\text{\emph{k}}}$
if $x\in I_{k}$ where $c_{\text{\emph{k}}}\in\mathbb{C}$ and $|c_{\text{\emph{k}}}|=1$
for $k=1,..,m$. Let $I_{k}$ be defined by $I_{k}=[(k-1)/m,k/m)$
for $k=1,..,m-1$ and $I_{m}=[(m-1)/m,1]$. Denote the set of these
step functions by $T_{\text{\emph{m}}}([0,1])$. Every such step function
$s\in T_{\text{\emph{m}}}([0,1])$, $m\in\mathbb{N}$ generates a
Löwner chain because the function $p(w,t)=(1+s(x(t))w)/(1-s(x(t))w)$
where $x(t)=e^{-t}$ has positive real part for every $t\in[0,\infty)$(\cite{1},Theorem
3.4). As a consequence every $s\in T_{\text{\emph{m}}}([0,1])$ is
a generating function of a schlicht function $f_{\text{\emph{s}}}$
which can be embedded as the initial point $f(z,0)=f_{s}(z)$ into
a Löwner chain $f(.,t)$ that satisfies the differential equation
(2). $f_{\text{\emph{s}}}$ in turn generates a point $(a_{\text{\emph{2}}}(f_{\text{\emph{s}}}),..,a_{\text{\emph{n}}}(f_{\text{\emph{s}}}))\in\mathbb{V}_{n}$
through formula (7). Let $\mathbb{V}_{n}^{m}$ denote the subset of
$\mathbb{V}_{n}$ that is generated by step functions $s\in T_{\text{\emph{m}}}([0,1])$
in this way. \\
A property is said to hold almost everywhere(a.e.) on a set $X$
if it holds everywhere on $X$ except for a subset of measure zero.
The gaussian function $[x],\; x\in\mathbb{R}$ will also be used,
it denotes the largest integer $\leq x$.\\
The basic idea of the method is to consider step functions with
an equidistant partition of $m$ subintervals as generating functions
and to calculate the maximum of the functional within this subclass.
Thus a sequence of values for the functional is obtained and it has
to be shown that this sequence converges to the maximum of the functional
within the whole class S. This is done in the proof of Theorem 1 and
the method can therefore be described as a direct variational method
of calculating maximizing sequences. 
\begin{thm}
For $m,n\in\mathbb{N}$ let $\mathbb{V}_{n},\,\,\mathbb{V}_{n}^{m}$
and $T_{m}((0,1])$ be defined as above and let $\Phi:\mathbb{V}_{n}\mathbb{\rightarrow C}$
denote a continuous functional. Suppose that for every $m\in\mathbb{N}$
there exists a function $f_{m}\in$ S such that \begin{equation}
\mathit{sup}\{Re\Phi(a)\,|\, a\in\mathbb{V}_{n}^{m}\}=Re\Phi(a_{2}(f_{m}),..,a_{n}(f_{m}))\;\mathit{where}\;(a_{2}(f_{m}),..,a_{n}(f_{m}))\in\mathbb{V}_{n}^{m}.\end{equation}
Then \begin{equation}
\underset{m\rightarrow\infty}{\mathit{lim}}Re\Phi(f_{m})=\underset{f\in S}{\mathit{max}}\, Re\Phi(f).\end{equation}
Furthermore if $s_{m}\in T_{m}([0,1])$ denotes a generating function
of $f_{m}$, $m\in\mathbb{N}$ then there exists a subsequence $(s_{m(k)})_{k\in\mathbb{N}}$
of $(s_{m})$ and a function $\tilde{\mu}\in L^{\infty}([0,1])$ so
that $s_{m(k)}\rightarrow\tilde{\mu}$ weak-{*} as $k\rightarrow\infty$,
the sequence $(f_{m(k)})_{k\in\mathbb{N}}$ converges locally uniformly
on the unit disk to a function $\tilde{f}\in$ S that maximizes $Re\Phi$
and $\tilde{\mu}$ is a generating function of $\tilde{f}$ . \end{thm}
\begin{proof}
Let $\varepsilon>0${\Large{} }be arbitrary and let $\tilde{h}\in$
S denote a function that maximizes $Re\Phi$. Since the single-slit
mappings lie dense in S there exists a single slit mapping $h\in$
S, $h(z)=z+\sum_{j=2}^{\infty}a_{\text{\emph{j}}}z^{\text{\emph{j}}}$
such that \begin{equation}
|Re\Phi(a_{\text{\emph{2}}},..,a_{\text{\emph{n}}})-Re\Phi(a_{\text{\emph{2}}}(\tilde{h}),..,a_{\text{\emph{n}}}(\tilde{h}))|\leq|\Phi(a_{\text{\emph{2}}},..,a_{\text{\emph{n}}})-\Phi(a_{\text{\emph{2}}}(\tilde{h}),..,a_{\text{\emph{n}}}(\tilde{h}))|<\varepsilon/2\end{equation}
Since $h$ is a single slit mapping there exists a continuous generating
function $\mu:(0,1]\rightarrow\mathbb{C}$ of $h$. For $k\in\mathbb{N}$
and $x\in(0,1]$ define a sequence of step functions $t_{\text{\emph{k}}}$
by $t{}_{\text{\emph{k}}}(x)=\mu(([kx]+1)/k)$. If $x\in(0,1)$ is
arbitrarily chosen, then since $[kx]/k\leq x<([kx]+1)/k$ it follows
that\[
0<\frac{[kx]+1}{k}-x=|x-\frac{[kx]+1}{k}|\leq|\frac{[kx]+1}{k}-\frac{[kx]}{k}|=\frac{1}{k}\]
Hence for every $\tilde{\varepsilon}>0$ there exists a $N(x,\tilde{\varepsilon})\in\mathcal{\mathbb{N}}$
so that \[
|\mu(x)-t_{k}(x)|=|\mu(x)-\mu(\frac{[kx]+1}{k})|<\tilde{\varepsilon}\]
if $k>N(x,\tilde{\varepsilon})$, i.e. \begin{equation}
\underset{k\rightarrow\infty}{\mathit{lim}}t_{k}(x)=\mu(x)\end{equation}
for $x\in(0,1)$. For every $k\in\mathbb{N}$ the step function $t{}_{k}$
generates a uniquely determined function $h_{k}\in$ S, suppose that
$h_{k}(z)=z+\sum_{j=2}^{\infty}a_{\text{\emph{j}}}^{\text{\emph{(k)}}}z^{\text{\emph{j}}}.$
By (7) there exist functions $g_{\text{\emph{m}}}^{\text{\emph{(k)}}}(x)$
and $g_{\text{\emph{m}}}(x)$ for $m\geq2$ that have the representations\begin{equation}
g_{\text{\emph{m}}}^{\text{\emph{(k)}}}(x)=-\int_{\emph{0}}^{\text{\emph{x}}}\sum_{j=1}^{m-1}2js{}^{\text{\emph{m-j-1}}}g_{j}^{\text{\emph{(k)}}}(s)t_{k}(s)^{\text{\emph{m-j}}}ds\end{equation}
and, since $\mu$ is a generating function of $h$ \begin{equation}
g_{\text{\emph{m}}}(x)=-\int_{\emph{0}}^{\text{\emph{x}}}\sum_{j=1}^{m-1}2js{}^{\text{\emph{m-j-1}}}g_{j}(s)\mu(s)^{\text{\emph{m-j}}}ds.\end{equation}
Then the relations $g_{\emph{m}}^{\text{\emph{(k)}}}(1)=a_{\text{\emph{m}}}^{\text{\emph{(k)}}}=a_{m}(h{}_{k})$
and $g_{\emph{m}}(1)=a_{\text{\emph{m}}}=a_{m}(h)$ hold for $m\in\mathbb{N}$.
Applying the Lebesgue dominated convergence theorem, the product rule
for limits and (11) to (12) by induction yields the relations \begin{equation}
\underset{k\rightarrow\infty}{\mathit{lim}}g_{\emph{m}}^{\text{\emph{(k)}}}(x)=g_{\text{\emph{m}}}(x)\end{equation}
for every $x\in(0,1]$ and every $m\geq2$. Since $\Phi$ is continuous
(14) with $x=1$ implies that there exists a $N(\varepsilon)\in\mathbb{N}$
so that for $m>N(\varepsilon)$ \begin{equation}
|\mathit{\mathit{Re}}\Phi(a_{\text{\emph{2}}}^{\text{\emph{(m)}}},..,a_{\text{\emph{n}}}^{\text{\emph{(m)}}})-\mathit{\mathit{Re}}\Phi(a_{\text{\emph{2}}},..,a_{\text{\emph{n}}})|\leq|\Phi(a_{\text{\emph{2}}}^{\text{\emph{(m)}}},..,a_{\text{\emph{n}}}^{\text{\emph{(m)}}})-\Phi(a_{\text{\emph{2}}},..,a_{\text{\emph{n}}})|<\varepsilon/2.\end{equation}
On the other hand since $(a_{\text{\emph{2}}}^{\text{\emph{(m)}}},..,a_{\text{\emph{n}}}^{\text{\emph{(m)}}})\in\mathbb{V}_{n}^{m}$
the supposition (8) implies that for every $m\in\mathbb{N}$ there
exists a function $f_{m}\in$ S such that \begin{equation}
\mathit{\mathit{Re}}\Phi(a_{\text{\emph{2}}}^{\text{\emph{(m)}}},..,a_{\text{\emph{n}}}^{\text{\emph{(m)}}})\leq\mathit{\mathit{Re}}\Phi(a_{\text{\emph{2}}}(f_{\text{\emph{m}}}),..,a_{\text{\emph{n}}}(f_{\text{\emph{m}}}))\end{equation}
and $(a_{\text{\emph{2}}}(f_{\text{\emph{m}}}),..,a_{\text{\emph{n}}}(f_{\text{\emph{m}}}))\in\mathbb{V}_{n}^{m}$.
But $\tilde{h}\in$ S was chosen as an extremal function that maximizes
$\mathit{Re}\Phi$ and hence (10),(15) and (16) imply that\begin{equation}
0\leq\mathit{\mathit{Re}}\Phi(\tilde{h})-\mathit{\mathit{Re}}\Phi(a_{\text{\emph{2}}}(f_{\text{\emph{m}}}),..,a_{\text{\emph{n}}}(f_{\text{\emph{m}}}))\leq\mathit{\mathit{Re}}\Phi(\tilde{h})-\mathit{\mathit{Re}}\Phi(a_{\text{\emph{2}}}^{\text{\emph{(m)}}},..,a_{\text{\emph{n}}}^{\text{\emph{(m)}}})<\varepsilon\end{equation}
for $m>N(\varepsilon)$. This proves (9). For an arbitrary $m\in\mathbb{N}$
since $(a_{2}(f_{m}),..,a_{n}(f_{m}))\in\mathbb{V}_{n}^{m}$ the schlicht
function $f_{m}$ possesses a generating step function $s_{m}\in T_{m}([0,1])$.
The sequence $(s_{\text{\emph{m}}})$ is uniformly bounded, $s_{m}\in L^{\infty}([0,1])$
and since $(L^{1}([0,1]))^{\text{*}}=L^{\infty}([0,1])$ the sequential
Banach-Alaoglu Theorem asserts that the sequence $(s_{m})$ possesses
a subsequence $(s_{\text{\emph{m(k)}}})$ that converges weak-{*}
to a function $\tilde{\mu}\in L^{\infty}([0,1])$ as $k\rightarrow\infty$.
Hence\begin{eqnarray*}
\underset{k\rightarrow\infty}{\mathit{lim}}a_{n}(f_{m(k)}) & = & \underset{k\rightarrow\infty}{\mathit{lim}}\left\{ -\int_{\emph{0}}^{\text{\emph{1}}}\sum_{j=1}^{n-1}2jt^{\text{\emph{n-j-1}}}g_{j}(t)s_{m(k)}(t)^{\text{\emph{n-j}}}dt\right\} \\
 & = & -\int_{\emph{0}}^{\text{\emph{1}}}\sum_{j=1}^{n-1}2jt^{\text{\emph{n-j-1}}}g_{j}(t)\tilde{\mu}(t)^{\text{\emph{n-j}}}dt\\
 & = & a_{n}(\tilde{f})\end{eqnarray*}
for every $n\in\mathbb{N}$, $n\geq2$ and $\tilde{\mu}$ is a generating
function of $\tilde{f}\in$ S. Since S is a normal family the subsequences
$(s_{\text{\emph{m(k)}}})$ and $(f_{\text{\emph{m(k)}}})$ can be
chosen so that $(f_{\text{\emph{m(k)}}})$ converges locally uniformly
on $\mathbb{D}$ to $\tilde{f}\in$ S as $k\rightarrow\infty$. This
proves the Theorem. 
\end{proof}
Theorems similar to Theorem 1 were given by Tammi(\cite{3}, chapters
3.3 and 3.4). Tammi considered bounded schlicht functions, the approach
and proof given here is somewhat different.\\
For Theorem 1 to be useful explicit formulae of the coefficients
of schlicht functions that are generated by step functions are needed.
Lemma 1 provides these representations for the second, third and fourth
coefficients.
\begin{lem}
Let $n\in\mathbb{N}$ be arbitrary, with the notation of $T_{n}([0,1])$
as above let $s\in T_{n}([0,1])$, $s:[0,1]\mathbb{\rightarrow C}$
be defined by $s(x)=c_{\text{\emph{k}}}$ if $x\in I_{k}$ where $c_{\text{\emph{k}}}\in\mathbb{C}$,
for $k=1,..,n$ and let $c_{\text{\emph{0}}}\equiv0$. Suppose that
$s$ is a generating function of the schlicht function $f(z)=z+\sum_{k=2}^{\infty}a_{k}z^{k}$
. Then\begin{equation}
a_{2}=-\frac{2}{n}\sum_{k=1}^{n}c_{k}\end{equation}
\begin{equation}
a_{3}=a_{2}^{2}-\sum_{k=1}^{n}\frac{2k-1}{n^{2}}c_{k}^{2}\end{equation}
\begin{equation}
a_{4}=3a_{2}a_{3}-2a_{2}^{3}-\frac{2}{n^{3}}\sum_{k=1}^{n}c_{k}^{2}\{k^{\text{\emph{2}}}c_{k}+(2k-1)\sum_{j=0}^{k-1}c_{j}\}.\end{equation}
\end{lem}
\begin{proof}
For $x\in(0,1)$ the closed form representation $s(x)=c_{[nx]+1}$
holds and $s(1)=c_{n}$. Then (7) yields\begin{equation}
g_{2}(x)=-\int_{0}^{x}2s(t)dt=-\frac{2}{n}\sum_{k=0}^{[nx]}c_{k}-2(x-\frac{[nx]}{n})c_{[nx]+1}\end{equation}
Taking the limit $\underset{x\rightarrow1-}{\mathit{lim}}\, g_{2}(x)$
in (21) proves (18). To prove (19) substitute $g_{2}'(t)=-2s(t)$
(equation (6) with $n=2$) in (7). This yields\[
g_{3}(1)=\int_{0}^{1}2g_{2}(t)g_{2}'(t)-\frac{1}{2}tg_{2}'(t)^{2}dt=g_{2}(1)^{2}-\sum_{k=1}^{n}\int_{(k-1)/n}^{k/n}2tc_{k}^{2}dt=g_{2}(1)^{2}-\sum_{k=1}^{n}c_{k}^{2}\frac{2k-1}{n^{2}}\]
This proves (19). To prove (20) substitute $g_{2}'(t)=-2s(t)$ in
(7) for $n=4$. Then\begin{eqnarray}
g_{4}(1) & = & \int_{0}^{1}\frac{1}{4}t^{2}g_{2}'(t)^{3}-tg_{2}(t)g_{2}'(t)^{2}+3g_{3}(t)g_{2}'(t)dt\end{eqnarray}
Apply integration by parts to (22) and substitute $g_{3}'(t)=2g_{2}(t)g_{2}'(t)-\frac{1}{2}tg_{2}'(t)^{2}$
to obtain \begin{eqnarray}
g_{4}(1) & = & 3g_{3}(1)g_{2}(1)-2g_{2}(1)^{3}+\int_{0}^{1}\frac{1}{4}t^{2}g_{2}'(t)^{3}+\frac{1}{2}tg_{2}(t)g_{2}'(t)^{2}dt\end{eqnarray}
Now use (21) to evaluate the integral expression in (23) on the subinterval
$((k-1)/n,k/n)$ where $1\leq k\leq n$ \begin{eqnarray*}
\int_{(k-1)/n}^{k/n}\frac{1}{4}t^{2}g_{2}'(t)^{3}+\frac{1}{2}tg_{2}(t)g_{2}'(t)^{2}dt & = & -\int_{(k-1)/n}^{k/n}2t^{2}c_{k}^{3}+2tc_{k}^{2}(\frac{2}{n}\sum_{j=0}^{k-1}c_{j}+2(t-\frac{k-1}{n})c_{k})dt\end{eqnarray*}
\begin{eqnarray*}
 & = & -\int_{(k-1)/n}^{k/n}6t^{2}c_{k}^{3}+4tc_{k}^{2}(\frac{1}{n}\sum_{j=0}^{k-1}c_{j}-c_{k}\frac{k-1}{n})dt\\
 & = & -2(\frac{k^{3}-(k-1)^{3}}{n^{3}})c_{k}^{3}-\int_{(k-1)/n}^{k/n}4tc_{k}^{2}(\frac{1}{n}\sum_{j=0}^{k-1}c_{j}-\frac{k-1}{n}c_{k})dt\\
 & = & -2(\frac{3k^{2}-3k+1}{n^{3}})c_{k}^{3}-2\frac{k^{2}-(k-1)^{2}}{n^{3}}(c_{k}^{2}\sum_{j=0}^{k-1}c_{j}-(k-1)c_{k}^{3})\\
 & = & -2(\frac{3k^{2}-3k+1}{n^{3}})c_{k}^{3}+2\frac{(2k-1)(k-1)}{n^{3}}c_{k}^{3}-2\frac{2k-1}{n^{3}}c_{k}^{2}\sum_{j=0}^{k-1}c_{j}\\
 & = & -2c_{k}^{2}(\frac{k^{2}}{n^{3}}c_{k}+\frac{2k-1}{n^{3}}\sum_{j=0}^{k-1}c_{j}).\end{eqnarray*}
Summing these expressions and substituting into (23) yields (20). 
\end{proof}
Lemma 1 allows to express the real and imaginary parts of the second,
third and fourth coefficients in (18)-(20) as trigonometric polynomials,
since the $c_{k}$ have representations $c_{k}=\mathit{cos}\varphi_{k}+i\mathit{sin}\varphi_{k}$
for $k=1,..,m$. As a consequence the real part $\mathit{\mathit{Re}}\Phi:\mathbb{V}_{n}^{m}\rightarrow\mathbb{R}$
of the continuous functional $\Phi$ restricted to $\mathbb{V}_{n}^{m}$
can be expressed as a continuous function $\Theta_{m}:[0,2\pi]^{m}\rightarrow\mathbb{R}$
$\Theta_{m}(\varphi_{1},..,\varphi_{m})=\mathit{\mathit{Re}}\Phi(a_{2},..,a_{n})$.
Therefore the Heine-Borel Theorem ensures that $\Theta_{m}$ attains
its maximum on $[0,2\pi]^{m}$. It follows that condition (8) in Theorem
1 holds for every continuous coefficient functional of the second,
third and fourth coefficients.\\
Maximizing sequences were calculated, the partition of the interval
$[0,1]$ was always chosen as an equipartition with a varying number
of subintervals $m$. A method of successive refinements was used
to calculate the maxima, the computer program used for the calculations
is available from the author. Comparision with well-known values of
functionals(Fekete-Szegö theorem, third column, Table 1) indicates
that by the method used here the functionals can be calculated with
a precision of approximately $0.000005$. Table 1 shows the results
for some functionals.\\
\emph{\noun{Table 1:}} Values of some functionals depending on
the number $m$ of subintervals

\begin{tabular}{|c|c|c|c|c|}
\hline 
 & $|\gamma_{1}|^{2}+2|\gamma_{2}|^{2}-\frac{3}{2}$ & $|\gamma_{1}|^{2}+2|\gamma_{2}|^{2}+3|\gamma_{3}|^{2}-\frac{11}{6}$ & $\frac{1}{2}|a_{3}-\frac{1}{4}a_{2}^{2}|$ & $|\frac{1}{2}a_{4}-\frac{1}{4}a_{3}a_{2}+\frac{1}{16}a_{2}^{3}|$\tabularnewline
\hline
\hline 
$m=50$ & $0.034815$ & $0.029473$ & $1.013393$ & $1.006727$\tabularnewline
\hline 
$m=100$ & $0.034845$ & $0.029566$ & $1.013412$ & $1.006755$\tabularnewline
\hline 
$m=200$ & $0.034853$ & $0.029591$ & $1.013414$ & $1.006762$\tabularnewline
\hline 
$m=400$ & $0.034854$ & $0.029596$ & $1.013415$ & $1.006763$\tabularnewline
\hline
\end{tabular}

Here the $\gamma_{k}$, $k=1,2,3$ denote the logarithmic coefficients
of schlicht functions(see \cite{1}, chapter 5 for a definition of
the logarithmic coefficients). The first two functionals of Table
1 are associated with the Milin-constant(see \cite{1}, chapter 5
for a definition of these functionals and their relation to the Milin-constant),
the third functional of Table 1 is the modulus of the fifth coefficient
of odd schlicht functions and the fourth functional is the modulus
of the seventh coefficient of odd schlicht functions. In particular
the calculations reveal that Leemans result(\cite{4}) that $1090/1083\thickapprox1.006463$
is the maximum of the modulus of the seventh coefficient of odd schlicht
functions which are typically real does not extend to the whole class
of schlicht functions. A numerical disproof within the limits of machine
precision and rounding errors can already be given for $m=20.$ If
the $c_{k}$ for $k=1,..,m$ from Lemma 1 are given by $c_{k}=\mathit{cos}\varphi_{k}+i\,\mathit{sin}\varphi_{k}$
then the generating step function given in Table 2(in terms of $\varphi_{k}$)
provides a counter example.\\
\noun{Table 2:} Angles in Radians(m=20) 

\begin{tabular}{|c|c|c|c|c|c|c|c|c|c|c|}
\hline 
$k$ & $1$ & $2$ & $3$ & $4$ & $5$ & $6$ & $7$ & $8$ & $9$ & $10$\tabularnewline
\hline
\hline 
$\varphi_{k}$ & $3,180$ & $3,185$ & $3,189$ & $3,195$ & $3,203$ & $3,212$ & $3,224$ & $3,241$ & $3,268$ & $3,315$\tabularnewline
\hline
\end{tabular}

\begin{tabular}{|c|c|c|c|c|c|c|c|c|c|c|}
\hline 
$k$ & $11$ & $12$ & $13$ & $14$ & $15$ & $16$ & $17$ & $18$ & $19$ & $20$\tabularnewline
\hline
\hline 
$\varphi_{k}$ & $3,536$ & $3,835$ & $2,770$ & $2,686$ & $2,600$ & $2,306$ & $3,000$ & $3,039$ & $3,062$ & $3,078$\tabularnewline
\hline
\end{tabular}

The value of the functional $|\frac{1}{2}a_{4}-\frac{1}{4}a_{3}a_{2}+\frac{1}{16}a_{2}^{3}|$
can be easily calculated by substituting the values of the $c_{k}$,
$k=1,..,20$ into (18)-(20), it is approximately $1,006491>1090/1083\thickapprox1.006463$.\\
\\
Numerical evidence indicates that the maximum of the functional
in the first column of Table 1 might be attained by a typically real
function. The next Theorem supports this assumption and provides the
best lower bound for the Milin-constant currently known.
\begin{thm}
Let T denote the class of typically real functions and S the class
of schlicht functions. Then \[
\underset{f\in S\cap T}{\mathit{max}}\{|\gamma_{1}|^{2}-1+2|\gamma_{2}|^{2}-\frac{1}{2}\}=2\lambda_{0}^{4}e^{-4\lambda_{0}}+(3\lambda_{0}^{2}+2\lambda_{0}+1)e^{-2\lambda_{0}}-1\thickapprox0.03485611\]
where $\lambda_{0}$ is the zero of the equation $0=4e^{-2\lambda}(\lambda^{2}-\lambda^{3})-3\lambda+1$
in $[0,\infty)$, i.e. $\lambda_{0}=0.39004568.$...\end{thm}
\begin{proof}
The first and second logarithmic coefficients $\gamma_{1}$ and $\gamma_{2}$
by (3) satisfy the equations $\dot{\gamma}_{1}(t)=\gamma_{1}(t)+\kappa(t)$
and $\dot{\gamma}_{2}(t)=2\gamma_{2}(t)+2\gamma_{1}(t)\kappa(t)+\kappa(t)^{2}$.
Then\[
\gamma_{1}(t)e^{-t}=-\int_{t}^{\infty}\kappa(s)e^{-s}ds\]
and the second equation is equivalent to the equation\[
\frac{d}{dt}(\gamma_{2}(t)e^{-2t})=-2(\gamma_{1}(t)e^{-t})\frac{d}{dt}(\gamma_{1}(t)e^{-t})+e^{-2t}\kappa(t)^{2}.\]
Let $\gamma_{1}(0)=\gamma_{1}$ and $\gamma_{2}(0)=\gamma_{2}$ then
integrating the last equation yields\[
\gamma_{2}=\gamma_{1}^{2}-\int_{0}^{\infty}e^{-2t}\kappa(t)^{2}dt\]
Suppose that $\mathit{Im}a_{2}=0$ and $\mathit{Im}a_{3}=0,$ then
also $\mathit{Im}\gamma_{2}=0$ and $\mathit{Im}\gamma_{1}=0$ and
\[
2|\gamma_{2}|{}^{2}+|\gamma_{1}|^{2}=2\{\gamma_{1}^{2}-Re\int_{0}^{\infty}e^{-2t}\kappa(t)^{2}dt\}^{2}+\gamma_{1}^{2}.\]
Let $Re\kappa(t)=\mathit{cos}\vartheta(t)$ then \[
2|\gamma_{2}|^{2}+|\gamma_{1}|^{2}=2\{\gamma_{1}^{2}-\int_{0}^{\infty}e^{-2t}\mathit{cos}2\vartheta(t)dt\}{}^{2}+\gamma_{1}^{2}\]
\[
=2\{\gamma_{1}^{2}+\frac{1}{2}-2\int_{0}^{\infty}(e^{-t}\mathit{cos}\vartheta(t))^{2}dt\}^{2}+\gamma_{1}^{2}.\]
If $\int_{0}^{\infty}(e^{-t}\mathit{cos}\vartheta(t))^{2}dt=(\lambda+\frac{1}{2})e^{-2\lambda}$
then the Valiron-Landau Lemma(\cite{1},chapter 3) yields $\gamma_{1}^{2}\leq(\lambda+1)^{2}e^{-2\lambda}$.
Hence\begin{equation}
2|\gamma_{2}|{}^{2}+|\gamma_{1}|^{2}\leq2(\lambda^{2}e^{-2\lambda}+\frac{1}{2})^{2}+(\lambda+1)^{2}e^{-2\lambda}=F(\lambda).\end{equation}
A necessary condition that $F(\lambda)$ attains its maximum is $0=4e^{-2\lambda}(\lambda^{2}-\lambda^{3})-3\lambda+1$.
The solution of the equation is $\lambda_{0}\thickapprox0.39004568$.
The existence of a schlicht function for which equality holds in (24)
is proved almost identical as in the proof of Theorem 3.22 in \cite{1}.
Since the second and the third coefficients were chosen to be real
Goodman's Theorem(\cite{5}) ensures the existence of a typically
real function for which equality holds in (24). 
\end{proof}
Furthermore, since numerical calculation of the Milin-Functional $\delta_{3}:=|\gamma_{1}|^{2}-1+2|\gamma_{2}|^{2}-\frac{1}{2}+3|\gamma_{3}|^{3}-\frac{1}{3}$
for $n=3$ shows that $\delta_{3}<0.03$ (see Table 1)the value given
in Theorem 2 might well be the Milin-constant.\\

\end{document}